\documentclass[11pt]{article}

\usepackage[T1]{fontenc}
\usepackage[backend=bibtex]{biblatex}

\usepackage[utf8]{inputenc}
\usepackage{amscd}
\usepackage{amsfonts}
\usepackage{amsmath}
\usepackage{amssymb}
\usepackage{amsthm}
\usepackage{hyperref}
\usepackage{mathtools}

\usepackage[normalem]{ulem} 
\usepackage{xcolor}

\theoremstyle{definition}
\newtheorem{remark}{Remark}
\newtheorem{property}{Property}
\theoremstyle{plain}
\newtheorem{lemma}{Lemma}
\newtheorem{theorem}{Theorem}
\newtheorem{result}{Result}

\newtheorem{corollary}{Corollary}

\DeclareMathOperator{\trace}{Tr}
\DeclareMathOperator{\PGL}{PGL}
\DeclareMathOperator{\PTL}{P\Gamma{L}}
\DeclareMathOperator{\PG}{PG}
\DeclareMathOperator{\AG}{AG}
\DeclareMathOperator{\F}{\mathbb{F}}
\DeclareMathOperator{\U}{\mathcal{U}_{BT}}

\date{}

\title{On the Equivalence, Stabilisers, and Feet \\of Buekenhout-Tits
	Unitals}
\author{Jake Faulkner\and Geertrui Van de Voorde\thanks{This author is supported by the Marsden Fund Council administered by the Royal Society of New Zealand.}}

\begin{document}

\maketitle
\begin{abstract}
	This paper addresses a number of problems concerning Buekenhout-Tits unitals
	in \(\PG(2, q^2)\), where \(q = 2^{2e + 1}\) and \(e \geq 1\). We show that all
	Buekenhout-Tits unitals are \(\PGL\)-equivalent (addressing an open problem
	in~\cite{barwick09_unital_projec_planes}), explicitly describe their
	\(\PTL\)-stabiliser (expanding Ebert's work
	in~\cite{ebert97_bueken_tits_unital}), and show that lines meet the feet of
	points not on \(\ell_{\infty}\) in at most four points. Finally, we show that feet
	of points not on \(\ell_{\infty}\) are not always a \(\{0, 1, 2, 4\}\)-set, in
	contrast to what happens for Buekenhout-Metz
	unitals~\cite{abarz_metz_feet}.
\end{abstract}


\section{Introduction}
\subsection{Background}

Let \(\PG(2, q^2)\) denote the Desarguesian projective plane over the finite field with $q^2$ elements, $\F_{q^2}$, where $q$ is a prime power. A \textit{unital} \(U\) in \(\PG(2, q^2)\) is a set of \(q^{3} + 1\) points  such
that every line of \(\PG(2, q^{2})\) meets \(U\) in 1 or \(q + 1\) points. The \textit{classical} or \textit{Hermitian} unital, usually denoted by $\mathcal{H}(2,q^2)$, arises by taking the absolute points of a non-degenerate Hermitian polarity. Each point $P$ not lying on a unital $U$, lies on \(q + 1\) tangent lines to \(U\); the \(q + 1\)
points of \(U\) whose tangent lines contain \(P\) are called the \textit{feet} of \(P\),
and are denoted by \(\tau_P(U)\).

It is well-known that \(\PG(2, q^2)\) can be modelled using a Desarguesian line
spread of \(\PG(3, q)\) embedded in \(\PG(4, q)\) via the \textit{Andr\'{e}/Bruck-Bose
(ABB)} construction. A wide class of unitals in $\PG(2,q^2)$, called \textit{Buekenhout unitals},
arise as follows from the ABB construction; starting in \(\PG(4, q)\) fixing a hyperplane
\(\Sigma\), and a Desarguesian spread of \(\Sigma\), we take any ovoidal cone
\(\mathcal{C}\) such that \(\mathcal{C} \cap \Sigma\) is a spread line of $\Sigma$. Then in
\(\PG(2, q^2)\), \(\mathcal{C}\) gives rise to a unital \(U\). If the base of
\(\mathcal{C}\) is an elliptic quadric, the unital is called a \textit{Buekenhout-Metz
	unital}. The family of Buekenhout-Metz unitals contains the Hermitian unitals, but there are many non-equivalent Buekenhout-Metz unitals (see \cite{Baker1992}, \cite{Ebert1992}).
	 If \(q = 2^{2e + 1}\), \(e \geq 1\), and the base of \(\mathcal{C}\) is
a Tits ovoid, the unital is a called a \textit{Buekenhout-Tits unital}. For more
information on unitals and their constructions, see~\cite{barwick09_unital_projec_planes}.

Unitals can be characterised by the combinatorial properties of their feet. It is easy to see that for the classical unital $\mathcal{H}(2,q^2)$, the feet of a point not on the unital are always collinear.
Thas~\cite{Thas1992} showed the converse, namely, that a unital \(U\) is classical if and only if for all points, not on $U$, the feet are collinear. This was improved by Aguglia and Ebert \cite{MR1888419} who showed that a unital \(U\) is classical if and only if there exist two tangent
lines \(\ell_{1}, \ell_{2}\) such that for all points \(P \in (\ell_1 \cup \ell_2) \setminus U\) the feet of \(P\) are collinear.
It is known (see e.g. \cite{barwick09_unital_projec_planes}) that if \(U\) is a non-classical Buekenhout-Metz
unital, the feet of a point \(P \notin U\) are collinear if and only if
\(P \in \ell_{\infty}\). Furthermore, it is shown in \cite{abarz_metz_feet} that if \(U\) is Buekenhout-Metz unital, a line
meets the feet of a point \(P \notin \ell_{\infty}\) in either $0$, $1$, $2$, or $4$
points. Ebert~\cite{ebert97_bueken_tits_unital} showed for a Buekenhout-Tits
unital, the feet of \(P \notin U\) are collinear if and only if
\(P \in \ell_{\infty}\). It is then natural to ask how a line may meet the feet
of a point \(P \notin \ell_{\infty}\) for Buekenhout-Tits unitals. We will answer this question in Theorem \ref{thm:bt-unital-line-meet-feet}.

Many characterisations of unitals make use of their stabilisers in $\PGL$, resp. $\PTL$.
In~\cite{cossidente2000group} it is shown that a unital is classical if its stabiliser contains a cyclic group of order \(q^{2} - q + 1\). Several other
characterisations of unitals by their stabiliser group are listed
in~\cite{barwick09_unital_projec_planes}. In~\cite{ebert97_bueken_tits_unital},
Ebert determined the stabiliser of a Buekenhout-Tits unital in \(\PGL(3, q^{2})\) (see Result \ref{result:ebert-stab}). We will extend this work in this paper.

\subsection{Summary of this paper}
In this paper we present three main results:
\begin{enumerate}
	\item We show that all Buekenhout-Tits unitals are projectively equivalent (see Theorem \ref{cor:proj-equivalence}).
	      This addresses an open problem in~\cite{barwick09_unital_projec_planes}, and is
	      alluded to in~\cite{feng19_exist_onan_config_ovoid_bueken} (see
	      Remark~\ref{rem:equivalence}).
	\item A description of the full stabiliser group of a Buekenhout-Tits unital
	      in \(\PTL(3, q^{2})\) (see Theorem \ref{thm:bt-unital-pgammal-stab}). Ebert~\cite{ebert97_bueken_tits_unital} only provides a
	      description of stabiliser of the Buekenhout-Tits unital in \(\PGL\)
	      (Result~\ref{result:ebert-stab}). The stabiliser of the classical unital in
	      \(\PTL(3, q^2)\) is \(\mathrm{P\Gamma U}(3, q^2)\), and the stabiliser of the
	      Buekenhout-Metz unital in \(\PTL(3, q^{2})\) is described in~\cite{Ebert1992}
	      for \(q\) even and~\cite{Baker1992} for \(q\) odd.
	\item If \(U\) is a Buekenhout-Tits unital, then a line \(\ell\) meets the
	      feet of a point \(P \notin (\ell_{\infty} \cup U)\) in at most 4 points.
	      Moreover, there exists a point \(P\) and line \(\ell\) such that the feet of
	      \(P\) meet \(\ell\) in exactly three points (see Theorem \ref{thm:bt-unital-line-meet-feet}). This highlights a difference
	      between Buekenhout-Metz unitals and Buekenhout-Tits unitals. It also
	      solves an open problem posed by Aguglia and Ebert~\cite{MR1888419} and later listed
	      in~\cite{barwick09_unital_projec_planes}.
\end{enumerate}

\subsection{Coordinates for a Buekenhout-Tits unital}\label{sec:preliminaries}
In~\cite{ebert97_bueken_tits_unital}, Ebert derives coordinates for a
Buekenhout-Tits unital $\U$ in \(\PG(2, q^{2})\), $q=2^{2e+1}$. Pick
\(\epsilon \in \mathbb{F}_{q^2}\) such that \(\epsilon^q = \epsilon + 1\), and
\(\epsilon^2 = \epsilon + \delta\) for some \(1 \neq \delta \in \mathbb{F}_q\)
with absolute trace equal to one. Then the following set of points in
\(\PG(2, q^2)\) is a Buekenhout-Tits unital,
\begin{equation}
	\label{eq:bt-definition}
	\U = \{(0, 0, 1)\} \cup \{P_{r, s, t} = (1, s + t \epsilon, r + (s^{\sigma + 2} + t^{\sigma} + st)\epsilon)\,|\, r, s, t \in \mathbb{F}_{q}\},
\end{equation}
where \(\sigma = 2^{e + 1}\) has the property that $\sigma^2$ induces the automorphism $x\mapsto x^2$ of $\F_{q^2}$. In addition, it can be verified that \(\sigma + 1\), \(\sigma + 2\), \(\sigma - 1\), and \(\sigma - 2\) all
induce permutations of \(\mathbb{F}_q\) with inverses induced by \(\sigma - 1\),
\(1 - \sigma/2\), \(\sigma + 1\) and \(-(\sigma/2 + 1)\) respectively.

The following theorem describes the group of projectivities (that is, elements of $\PGL(3,q^2)$) stabilising \(\U\).
\begin{result}\cite[Theorem 4]{ebert97_bueken_tits_unital}\label{result:ebert-stab}
	Let \(G=\PGL(3,q^2)_{\U}\), $q=2^{2e+1}$, be the group of projectivities stabilising the Buekenhout-Tits unital \(\U\). Then \(G\)
	is an abelian group of order \(q^{2}\), consisting of the projectivities induced by
	the matrices
	\begin{equation}
		\label{eq:Muv}
		M_{u,v} = \left\{\index{$\begin$}\begin{bmatrix}
			1 & u \epsilon & v + u^{\sigma} \epsilon \\
			0 & 1          & u + u \epsilon          \\
			0 & 0          & 1
		\end{bmatrix}\,\middle|\,u,v \in \mathbb{F}_{q}\right\},
	\end{equation}
	where $\sigma=2^{e+1}$ and matrices act on the homogeneous coordinates of points by
	multiplication from the right.
\end{result}
\section{On the Projective Equivalence of Buekenhout-Tits Unitals}
In this section, we show that all Buekenhout-Tits unitals are projectively
equivalent to the unital $\U$ given in equation~\eqref{eq:bt-definition}.
\begin{remark}\label{rem:equivalence}
	The authors of~\cite{feng19_exist_onan_config_ovoid_bueken} give this result
	without proof and state it can be derived using the same techniques employed by Ebert
	in~\cite{ebert97_bueken_tits_unital}. Ebert however, lists the equivalence of
	Buekenhout-Tits unital as an open problem
	in~\cite{barwick09_unital_projec_planes} which appeared about ten years after his original
	paper~\cite{ebert97_bueken_tits_unital}.
\end{remark}
It is easy to see that the Buekenhout-Tits unital $\U$ is tangent to the
line \(\ell_{\infty} : x = 0\) in the point \(P_{\infty} = (0, 0, 1)\). From the
ABB construction it follows that \(P_{\infty}\) has the following property with respect to $\U$.
\begin{property}\label{prop:subline-property}
	Given any unital \(U\), a point \(P \in U\) has
	Property~\ref{prop:subline-property} if all secant lines through \(P\) meet
	\(U\) in Baer sublines.
\end{property}
It is shown in~\cite{Barwick2001} that if two different points of \(U\) have Property~\ref{prop:subline-property}, then $U$ is classical. Hence, the point $P_\infty$ is the unique point of
$\U$ admitting this property.
We will count all Buekenhout-Tits unitals tangent to
\(\ell_{\infty}\) at a point \(P_{\infty}\) having
Property~\ref{prop:subline-property}.

\begin{lemma}\label{thm:unitals-equivalent-in-plane}
	There are \(q^{4} {(q^2 - 1)}^2\) unitals projectively equivalent to \(\U\)
	in \(\PG(2, q^2)\) tangent to \(\ell_{\infty}\, :\,x = 0\), and containing the
	point \(P_{\infty} = (0, 0, 1)\) with Property~\ref{prop:subline-property}.
\end{lemma}
\begin{proof}
  First note that any projectivity mapping $\U$ to a unital tangent to
  $\ell_\infty$ in $P_\infty$ necessarily is contained in the group \(H\) of
  projectivities fixing \(\ell_{\infty}\) line-wise and \(P_{\infty}\) point-wise.
  The elements of $H$ are induced by all matrices of the following form,
  \begin{equation*}
	\begin{bmatrix}
	  1 & x_{12} & x_{13} \\
	  0 & x_{22} & x_{23} \\
	  0 & 0      & x_{33}
	\end{bmatrix},
  \end{equation*}
  where \(x_{22} x_{33} \neq 0\) and matrices act on homogeneous coordinates by
  multiplication on the right. It follows that \(|H|={(q^2 - 1)}^2 q^6\).
  Furthermore, from the description of \(G=\PGL(3,q^2)_{\U}\) in
  Result~\ref{result:ebert-stab}, we know that \(H_{\U}=G\), and hence,  \(H_{\U}\)
  has order \(q^{2}\). By the orbit-stabiliser theorem, we find that there are
  \({(q^2 - 1)}^2 q^4\) unitals in the orbit of \(\U\) under \(H\).
\end{proof}

Consider \(\PG(2, q^{2})\) modelled using the ABB construction with fixed
hyperplane \(H_{\infty}\). Let \(p_{\infty}\) be the spread line
corresponding to \(P_{\infty}\). Then any Buekenhout-Tits unital \(U\) tangent
to \(\ell_{\infty}\) at \(P_{\infty}\) with Property~\ref{prop:subline-property}
corresponds uniquely to an ovoidal cone \(\mathcal{C}\) meeting \(H_{\infty}\)
at \(p_{\infty}\).

\begin{lemma}\label{lem:numbercones}
	There are \(q^{4}{(q^{2} - 1)}^{2}\) ovoidal cones \(\mathcal{C}\) in \(\PG(4, q)\) with
	base a Tits ovoid, such that \(\mathcal{C}\) meets \(H_{\infty}\) in the spread element
	\(p_{\infty}\).
\end{lemma}
\begin{proof}
  Let \(V\) be a point on the line \(p_{\infty}\), and
  \(H \neq H_{\infty}\) a hyperplane not containing \(V\). Then, \(H\) meets
  \(H_{\infty}\) in a plane containing a point \(R \in p_{\infty} \setminus \{V\}\). Any ovoidal
  cone \(\mathcal{C}\) with vertex \(V\) and base a Tits ovoid, such that
  \(\mathcal{C}\) meets \(H_\infty\) precisely in \(p_\infty\), meets \(H\) in a
  Tits ovoid tangent to \(H \cap H_{\infty}\) at the point \(R\). We will count
  all cones of this form, for all \(V \in p_{\infty}\).

  Consider the pairs of planes \(\Pi\) and Tits ovoids \(\mathcal{O}\),
  \((\Pi, \mathcal{O})\), where \(\Pi, \mathcal{O} \subset H\) and \(\Pi\) is
  tangent to \(\mathcal{O}\). On the one hand, there are
  \(|\PGL(4, q)|/|\mathcal{O}_{\PGL(4, q)}| = {(q + 1)}^2 q^4 {(q - 1)}^2 {(q^2 + q + 1)}\)
  Tits ovoids in \(\PG(3, q)\), and each has \(q^{2} + 1\) tangent planes. On the
  other hand, \(\PGL(4, q)\) is transitive on hyperplanes of \(\PG(3, q)\), so
  each plane is tangent to the same number of Tits ovoids. It thus follows, that
  there
  are \[\frac{(q + 1)^2q^4{(q-1)}^2(q^2+q+1)(q^2+1) }{q^4 + q^3+q^2+q+1}= {(q-1)}^{2}q^4(q+1)(q^2+q+1)\]
  Tits ovoids tangent to \(H \cap H_{\infty}\) conitained in \(H\).

  Furthermore, since \({\PGL(4, q)}_{H \cap H_{\infty}}\) is transitive on
  points of \(H \cap H_{\infty}\), each point of \(H \cap H_{\infty}\) is
  contained in	the same number of Tits ovoids \(\mathcal{O}\), so it follows
  that the number of Tits ovoids tangent to \(H \cap H_{\infty}\) at
  \(R = p_{\infty} \cap H\) is \({(q - 1)}^2q^4(q + 1)\). Hence, there is an equal
  number of ovoidal cones with base a Tits ovoid, vertex \(V\), and meeting
  \(H_{\infty}\) at \(p_{\infty}\). As the choice of \(V\) was arbitrary, and
  there are \(q+1\) points on \(p_\infty\), there are \({(q^2 - 1)}^2q^4\) ovoidal
  cones with base a Tits ovoid, and meeting \(H_{\infty}\) at \(p_{\infty}\).
\end{proof}
\begin{theorem}\label{cor:proj-equivalence}
	All Buekenhout-Tits unitals in \(\PG(2, q^{2})\) are \(\PGL\)-equivalent.
\end{theorem}
\begin{proof}
	From Lemmas \ref{thm:unitals-equivalent-in-plane} and \ref{lem:numbercones}, we see that
	the number of ovoidal cones with vertex a Tits ovoid, tangent to \(H_{\infty}\) at \(p_{\infty}\) is
	equal to the number of Buekenhout-Tits unitals that are \(\PGL\) equivalent to
	\(\U\) and tangent to \(l_{\infty}\) at \(P_{\infty}\) with
	Property~\ref{prop:subline-property}. The result follows.
\end{proof}
\begin{corollary}\label{cor:bt-stab-equiv}
	Let \(U\) be a Buekenhout-Tits unital, then the projectivity group stabilising
	\(U\) is isomorphic to the group \(G\) in Theorem~\ref{result:ebert-stab}.
\end{corollary}
In showing that all Buekenhout-Tits unitals are projectively equivalent, we may
use \(\U\) to verify statements about general Buekenhout-Tits unitals.

\section{On the Stabiliser of the Buekenhout-Tits Unital}\label{sec:stab-buek-tits}
We now describe the stabiliser of the Buekenhout-Tits unital \(\U\) in \(\PTL(3, q^{2})\).
\begin{lemma}\label{lem:G-matrix-mult}
	Let \(M_{u, v}, M_{s, t}\) be matrices inducing collineations of \(G\) as defined in
	Result~\ref{result:ebert-stab}, then
	\(M_{u, v} M_{s, t} = M_{u + s, t + v + su \delta}\).
\end{lemma}
\begin{proof}
	Using equation~\eqref{eq:Muv}, we find
	\begin{align}
		M_{u, v} M_{s, t} =
		\begin{bmatrix}
			1 & (s + u)\epsilon & (t + v + su \delta) + {(s + u)}^{\sigma} \\
			0 & 1               & (u + s) + (u + s)\epsilon                \\
			0 & 0               & 1
		\end{bmatrix}.
	\end{align}
	Thus, we have \(M_{u,v} M_{s, t} = M_{u + s, t + v + su \delta}\).
\end{proof}
\begin{corollary}\label{cor:G-matrix-order}
  The order of any collineation of \(G\) induced by a matrix \(M_{u, v}\) as
  defined in Result~\ref{result:ebert-stab} is four if and only if \(u \neq 0\),
  and two if and only if \(u = 0\) and \(v \neq 0\).
\end{corollary}
\begin{proof} Firstly note that \(M_{0, 0} = I\). Direct calculation shows that \(M_{u,v}^2=M_{0,u^2\delta}\), \(M_{u,v}^3=M_{u,v+u^2\delta}\) and \(M_{u,v}^4=M_{0,0}\).
\end{proof}
\begin{corollary}\label{cor:bt-homography-group-simple}
	The stabiliser group \(G\) as defined in Result~\ref{result:ebert-stab} is isomorphic to \({(C_{4})}^{2e + 1}\).
\end{corollary}
\begin{proof} Recall from Result~\ref{result:ebert-stab} that $|G|=q^2=2^{4e+2}$.
	From Corollary~\ref{cor:G-matrix-order}, we have that
	\(G \equiv {(C_{4})}^{k}{(C_{2})}^{l}\) for some integers \(k, l\) such that
	\(2^{2k + l} = |G| = 2^{4e + 2}\), and hence,
	\begin{align}l = 2(e + 1 - k). \label{constraints}
	\end{align} Furthermore, we see that the number of elements of order four in $G$ is
	\(q^{2} - q\) as they correspond to all matrices $M_{u,v}$ with $u,v\in \F_q$ and $u\neq 0$. The number of elements of order four in a group isomorphic to
	\({(C_{4})}^{k}{(C_{2})}^{l}\) is \((4^{k} - 2^{k})2^{l}\).  Thus,
	\begin{equation}
		(4^{k} - 2^{k})2^{l}         = 4^{2e + 1} - 2^{2e + 1}.
	\end{equation}
	Using \eqref{constraints}, we find that \(k = 2e + 1\), and therefore \(G \equiv {(C_{4})}^{2e + 1}\).
\end{proof}

\begin{theorem}\label{thm:bt-unital-pgammal-stab}
	Let \(U\) be a Buekenhout-Tits unital in $\PG(2,q^2)$, $q=2^{2e+1}$, then the stabiliser group of \(U\) in
	\(\PTL(3, q^{2})\) is the order \(q^{2}(4e + 2)\) group \(GK\), where
	\(K\) is a cyclic subgroup of order \(16e + 8\) generated by
	\begin{equation}
		\psi  :\mathbf{x}\mapsto \mathbf{x}^{2} \begin{bmatrix}
			1 & 1                               & \epsilon                        \\
			0 & \delta^{\sigma/2}(1 + \epsilon) & \delta^{\sigma/2}(1 + \epsilon) \\
			0 & 0                               & \delta^{\sigma+1}
		\end{bmatrix}.
	\end{equation}

	(Here, $\mathbf{x}$ denotes the row vector containing the three homogeneous coordinates of a point, and $\mathbf{x}^2$ denotes its elementwise power.)
\end{theorem}
\begin{proof}
  From Lemma \ref{lem:numbercones}, we have that the number of Buekenhout-Tits
  unitals is \(q^4{(q^2 - 1)}^{2}\). Since all of those unitals are
  \(\PGL\)-equivalent by Theorem \ref{cor:proj-equivalence}, and
  \(\PGL(3, q^2) \triangleleft \PTL(3, q^2)\), we have that
	\begin{equation}
		q^4{(q^2 - 1)}^{2} = \frac{|\PGL(3, q^2)|}{|\PGL{(3, q^2)}_U|} = \frac{|\PTL(3, q^2)|}{|\PTL{(3, q^2)}_{U}|}.
	\end{equation}

	So \(\PTL(3, q^2)_U\) must have order \(q^{2}(4e + 2)\). Direct calculation
	shows that \(\psi\) stabilises $\U$. We have \(\psi^{4e + 2} \in G\) as
	\(\mathbf{x}^{2^{4e + 2}} = \mathbf{x}^{q^{2}} = \mathbf{x}\). Hence, \(|\psi| = (4e + 2)|\psi^{4e + 2}|\).
	From Corollary~\ref{cor:G-matrix-order}, it follows that
	\(|\psi^{4e + 2}| \in \{ 1, 2, 4\}\), with \(|\psi^{4e + 2}| = 4\) if and only
	if \(\psi^{4e + 2}\) is induced by \(M_{u, v}\) for some \(u \neq 0\). Hence,
	\(|\psi^{4e + 2}| = 4\) if and only if \(\psi^{4e + 2}(0, 1, 0) \neq (0, 1, 0)\)
	as \((0, 1, 0)M_{u, v} = (0,1,u + u \epsilon)\). Consider the point \((0, 1, z)\)
	for some arbitrary \(z \in \mathbb{F}_q\). Direct calculation shows that
	\(\psi(0, 1, z) = (0, 1, 1 + \mu z^2)\), where
	\(\mu = \frac{\delta^{\sigma + 1}}{\delta^{\sigma/2}(1 + \epsilon)} = \delta^{\sigma/2}\epsilon\).
	Thus,
	\begin{equation}
		\psi^{k}(0, 1, z) = (0, 1, \sum_{i=0}^{k}\mu^{2^{i} - 1} + zg(z))
	\end{equation}
	for some polynomial \(g(z)\) depending on \(k\). If \(z =
	0\) and \(k = 4e + 2\) we thus find:
	\begin{align}
		\psi^{4e + 2}(0, 1, 0) & = (0, 1, \sum_{i=0}^{4e + 2}\mu^{2^{i} - 1}) \\
		                       & = (0, 1, \frac{\trace(\mu)}{\mu}).
	\end{align}
	 Recall that $\epsilon^q=\epsilon+1$, so $\trace_{\mathbb{F}_{q^2}/\mathbb{F}_{q}}(\epsilon)=1$.
	We have that, $\trace_{\mathbb{F}_{q^2}/\mathbb{F}_{2}}(\delta^{\sigma/2}\epsilon)=\trace_{\mathbb{F}_{q}/\mathbb{F}_{2}}(\trace_{\mathbb{F}_{q^2}/\mathbb{F}_{q}}(\delta^{\sigma/2}\epsilon))=\trace_{\mathbb{F}_{q}/\mathbb{F}_{2}}(\delta^{\sigma/2}\trace_{\mathbb{F}_{q^2}/\mathbb{F}_{q}}(\epsilon))=\trace_{\mathbb{F}_{q}/\mathbb{F}_{2}}(\delta^{\sigma/2}) = 1.$
	Hence, \(\psi((0, 1, 0)) \neq (0, 1, 0)\), so \(|\psi^{4e + 2}| = 4\) and
	\(|\psi| = 16e + 8\). Let \(K = \langle\psi\rangle\), because
	\(|K \cap G| = 4\), it follows that \(|GK| = q^{2}(4e + 2)\) and thus
	\(GK = \PTL{(3, q^{2})}_U\).
\end{proof}

\section{On the Feet of the Buekenhout-Tits Unital}
The feet of the Buekenhout-Tits unital $\U$ are first described by Ebert
in~\cite{ebert97_bueken_tits_unital}. He shows that the feet of a point
\(P = (1, y_1 + y_2\epsilon, z_1 + z_2\epsilon)\) is the following set of
points:
\begin{multline} \label{feetex}
	\tau_P(\U) = \{(1, s + t\epsilon, s^2 + t^2\delta + st + y_1s + y_1t + y_2\delta{t} + z_1 + (s^{\sigma + 2} + t^{\sigma} + st)\epsilon)\\|\,s, t \in \mathbb{F}_q,\, s^{\sigma + 2} + t^{\sigma} + st = y_2s + y_1t + z_2\}.
\end{multline}
If the line \(\ell\) has equation \(\alpha x + y = 0\), where \(\alpha \in \mathbb{F}_{q^2}\), Ebert
shows that \(|\ell \cap \tau_P(\U)| \leq 1\). Otherwise, \(\ell\) has equation
\((a_1 + a_2\epsilon)x + (b_1 + b_2\epsilon)y + z = 0\), with $a_1,a_2,b_1,b_2\in \F_q$, and Ebert shows that \(\ell\)
meets \(\tau_P(\U)\) in the points \(P_{r, s, t} \in \U\), where \(r, s, t \in \F_q\) satisfy
\begin{align}
	s^2 + \delta t^2 + st + (y_1 + b_1) s + (y_1 + y_2 \delta + b_2 \delta) t + z_1 + a_1 & = 0,                            \label{eq:sys-1}       \\
	s^{\sigma + 2} + t^{\sigma} + st                                                      & = b_2 s + (b_1 + b_2) t + a_2,\label{eq:sys-2}         \\
	y_{2} s + y_{1} t + z_{2}                                                             & = b_{2} s + (b_{1} + b_{2})t + a_{2}. \label{eq:sys-3}
\end{align}
We will show that for all choices of points \(P \notin \ell_{\infty}\) and lines \(\ell\), \(|\tau_{P}(\U) \cap \ell| \leq 4\).

\begin{lemma}\label{lem:orb-reps}
	Let \(G\) be the group of projectivities stabilising the Buekenhout-Tits
	unital as described in Result~\ref{result:ebert-stab}. Then, the set of \(q^2 - q\)
	points
	\(\{P_{a, b} = (1, a, b \epsilon) \,|\, a,b \in \mathbb{F}_q,\, b \neq a^{\sigma + 2}\}\)
	are points from \(q^2 - q\) distinct point orbits of order \(q^2\) under \(G\).
\end{lemma}
\begin{proof}
	Suppose there exists a collineation of \(G\) induced by a matrix \(M_{u, v}\) such that \(P_{a, b} M_{u, v} = P_{c, d}\). Then,
	\begin{equation*}
		\left( 1, a, b \epsilon \right) \begin{bmatrix}
			1 & u \epsilon & v + u^{\sigma} \epsilon \\
			0 & 1          & u + u \epsilon          \\
			0 & 0          & 1
		\end{bmatrix} = \left( 1, c, d \epsilon \right).
	\end{equation*}
	However, it is clear that
	\(P_{a, b}M_{u, v} = \left( 1, a + u \epsilon, v + u^{\sigma}\epsilon + a \left( u + u \epsilon \right) + b\epsilon \right)\),
	so \(a + u \epsilon = c\). Therefore, \(a = c\) and \(u = 0\). If \(u = 0\),
	then \(v + b\epsilon = d \epsilon\), and we have \(b = d\). Hence,
	\(P_{a, b} = P_{c, d}\) and the lemma follows.
\end{proof}
There are \(q^4 - q^3 = q^{2} (q^2 - q)\) points of \(\PG(2, q)\) not on
\(\ell_{\infty}\) or \(\U\). By Lemma~\ref{lem:orb-reps}, each of these points
lies in the orbit of a point of the form \((1, a, b \epsilon)\). Therefore, in order to study the feet of a point $P$, we
may assume that the point \(P = (1, y_1 + y_2\epsilon, z_1 + z_2\epsilon)\) has
\(y_2 = z_1 = 0\).

The following lemma shows that the feet of a point $P = (1, y_1, z_2\epsilon)$ meets almost all lines in at most $2$ points.
\begin{lemma}\label{lem:easy-case}  Let $\ell:\alpha x+\beta y + z$ be a line in $\PG(2,q^2)$,
	where \(\alpha= a_{1} + a_{2} \epsilon\),
	\(\beta = b_{1} + b_{2} \epsilon\) and
	\(a_{1}, a_{2}, b_{1}, b_{2} \in \mathbb{F}_{q}\). Let
	\(P = (1, y_{1}, z_{2} \epsilon)\), with \(y_{1}, z_{2} \in \mathbb{F}_{q}\)
	such that \(z_{2} \neq y_{1}\). Unless \(b_{2} = 0\), \(y_{1} = b_{1}\) and
	\(a_{2} = z_{2}\), we have \(|\tau_P(\U) \cap \ell| \leq 2\).
\end{lemma}
\begin{proof} From the description given in \eqref{feetex}, we see that
	the points \(P_{r, s, t} \in \tau_{P}(\U)\) satisfy \(s^{\sigma + 2} + t^{\sigma} + st = y_{1}t + z_{2}\), and this equation has $q+1$ solutions. Substituting this
	into equation~\eqref{eq:sys-2}, the points \(P_{r, s, t} \in \tau_P(\U) \cap \ell\) have \(s, t\) satisfying
	\begin{align}
		s^{2} + \delta t^{2} + st + \left( y_{1} + b_{1} \right)s + \left( y_{1} + b_{2} \delta \right) t  + a_{1} & = 0 \label{eq:sys-1-easy-case}  \\
		b_{2} s + \left( y_{1} + b_{1} + b_{2}  \right)t + a_{2} + z_{2}                                           & = 0 \label{eq:sys-2-easy-case}  \\
		s^{\sigma + 2} + t^{\sigma} + st + y_{1}t + z_{2}                                                          & = 0. \label{eq:sys-3-easy-case}
	\end{align}
 Recall that the points $(1,s,t,s^{\sigma + 2}+t^\sigma+ st)$, where $s,t\in\F_q$ are the $q^2$ affine points of a Tits ovoid. Hence, \eqref{eq:sys-3-easy-case} represents an affine section of a Tits ovoid. Since it has $q+1$ points, it is an oval projectively equivalent to the translation oval \(\mathcal{D}_{\sigma} = \left\{(1, t, t^{\sigma})\,|\, t \in \mathbb{F}_q\right\}\).
	Unless \(b_{2} = 0\) and \(y_{1} = b_{1}\), equation~\eqref{eq:sys-2-easy-case}
	represents a line in \(\AG(2, q)\) which meets the oval \eqref{eq:sys-3-easy-case} in at most two points, so we have at most two solutions to the
	system. If \(b_{2} = 0\), \(y_{1} = b_{1}\), and \(a_{2} \neq z_{2}\), then
	equation~(\ref{eq:sys-2-easy-case}) has no solutions.
\end{proof}
\begin{remark}\label{rem:barwick-remark}
	Lemma~\ref{lem:easy-case} is a refinement of~\cite[Theorem 4.33]{barwick09_unital_projec_planes}, where Barwick and Ebert rework Ebert's
	earlier proof in~\cite{ebert97_bueken_tits_unital} that the feet of a point
	\(P \notin (\ell_{\infty} \cup \U)\) are not collinear. This reworked proof asserts
	that the feet cannot be collinear because the line given by
	equation~\eqref{eq:sys-2-easy-case} and the conic from
	equation~\eqref{eq:sys-1-easy-case} cannot have \(q + 1\) common solutions. However, we can see
	 that this logic is not complete, and leaves an interesting case to examine
	when equation~\eqref{eq:sys-2-easy-case} vanishes. Ebert's original proof
	in~\cite{ebert97_bueken_tits_unital} does not contain this error, instead
	arguing that equations~\eqref{eq:sys-1-easy-case} and~\eqref{eq:sys-3-easy-case}
	cannot have \(q + 1\) common solutions.
\end{remark}
It follows from Lemma~\ref{lem:easy-case} that the feet of a point
\(P \notin (\ell_{\infty} \cup \U)\) is a set of \(q + 1\) points such that every
line meets \(\tau_P(\U)\) in at most two points except for a set of \(q\)
concurrent lines.

To this end, assume that \(b_{2} = 0\), \(y_1 = b_1\) and \(a_2 = z_2\). In this case,
equation~\eqref{eq:sys-2-easy-case} vanishes. The system describing
\(\ell \cap \tau_P(\U)\) is thus
\begin{align}
	s^{2} + \delta t^{2} + st        & =  y_{1}  t + a_{1}\label{eq:sys-1-simple} \\
	s^{\sigma + 2} + t^{\sigma} + st & = y_{1} t + z_{2}\label{eq:sys-2-simple}.
\end{align}
The lines that produce these cases are the lines with dual coordinates
\([a_{1} + z_{2}\epsilon, y_{1}, 1]\). These lines are concurrent at the point
\((0, 1, y_{1})\) which lies on \(\ell_{\infty}\). We will show in Corollary \ref{cor:rough-bound} that these latter lines meet
\(\tau_P(\U)\) in at most four points. 

We require the following lemma, which adapts arguments found in~\cite[Lemma
	2.1]{ceria21_mds}.
\begin{lemma}\label{lem:nucleus-argument}
	Let \(\mathcal{O}\) be a translation oval in \(\PG(2, q)\) projectively
	equivalent to \(\mathcal{D}_{\sigma}\), and let \(\mathcal{C}\) be a
	non-degenerate conic. If the nucleus of \(\mathcal{O}\) is also the nucleus of
	\(\mathcal{C}\), then \(|\mathcal{O} \cap \mathcal{C}| \leq 4\).
\end{lemma}
\begin{proof}
	Without loss of generality we may take
	\(\mathcal{O} = \mathcal{D}_{\sigma}\), so that the nucleus of \(\mathcal{O}\)
	is \(N = (0, 1, 0)\). If \(N\) is also the nucleus of \(\mathcal{C}\), then
	\(\mathcal{C}\) is a conic of the following form,

	\begin{equation}\label{eq:conic-equation}
		a_1 x^2 + a_2 y^2 + a_3 z^2 + x z = 0,
	\end{equation}
	for some \(a_1, a_2, a_3 \in \mathbb{F}_q\) with \(a_{2} \neq 0\). Suppose
	that \((0, 0, 1) \notin \mathcal{C}\). Then \(a_3 \neq 0\), and the point
	\((1, t, t^{\sigma}) \in \mathcal{O}\) if and only if \(t\) satisfies
	\begin{equation}\label{eq:sigmas}
		a_1 + a_2 t^2 + a_3 t^{2 \sigma} + t^\sigma = 0,
	\end{equation}
	that is
	\begin{equation}
		0 = {\left( a_1 + a_2 t^{2} + a_3 t^{2 \sigma} + t^{\sigma} \right)}^{\sigma/2} = a_{1}^{\sigma/2} + a_2^{\sigma/2} t^\sigma + t^2 + t.
	\end{equation}
	Therefore,
	\begin{equation}
		t^\sigma = {\left( \frac{a_3}{a_2} \right)}^{2^{e}} t^{2} + \frac{1}{a_2^{2^e}} t + {\left(\frac{a_1}{a_2}\right)}^{2^{e}}
	\end{equation}
	and substituting into equation \eqref{eq:sigmas}, we find that this
	equation has at most four solutions. If instead \((0, 0, 1) \in \mathcal{C}\),
	then \(a_3 = 0\) and arguing as above we find that
	equation \eqref{eq:sigmas} has at most two solutions, so
	\(|\mathcal{O} \cap \mathcal{C}| \leq 3\).
\end{proof}
\begin{corollary}\label{cor:rough-bound}
	The feet of a point \(P \notin \left( \ell_{\infty} \cup \U \right)\) meet a
	line \(\ell\) in at most four points.
\end{corollary}
\begin{proof}From Lemma \ref{lem:easy-case}, we know we can restrict ourselves to the case $b_2=0,y_1=b_1,a_2=z_2$ which means we are looking at the
	points \(P_{r, s, t} \in \tau_P(\U) \cap \ell\) have \(s, t\)
	satisfying
	\begin{align}
		s^{2} + \delta t^{2} + st        & =  y_{1}  t + a_{1}\label{eq:conic} \\
		s^{\sigma + 2} + t^{\sigma} + st & = y_{1} t + z_{2}\label{eq:oval},
	\end{align}
	 where equation \eqref{eq:conic} represents a conic $\mathcal{C}$, and equation \eqref{eq:oval} represents an oval $\mathcal{O}$ in \(\AG(2, q)\). If the conic is degenerate, it's easy to see that the oval and conic have at most four points in common. So we may assume that the conic is non-degenerate.
  The nucleus of $\mathcal{C}$ is \(N = (y_1, 0, 1)\). We now show that \(N\) is the nucleus of the oval $\mathcal{O}$.   The line \(t = 0\) goes through $N$ and
	meets the oval~\eqref{eq:oval} when \(s^{\sigma + 2} = z_{2}\), which has one
	solution as \(\sigma + 2\) is a permutation of \(\mathbb{F}_{q}\). The line
	\(s + y_{1} = 0\) through $N$ meets the oval~\eqref{eq:oval} when
	\(t^{\sigma} = y^{\sigma + 2} + z_{2}\) which has one solution for $t$. Therefore, \(N\) is the nucleus, as it is the intersection of two tangent lines to the oval.
	It now follows from
	Lemma~\ref{lem:nucleus-argument} that 	equations~\eqref{eq:conic}
	and~\eqref{eq:oval} have at most four common solutions.
\end{proof}

We now show the existence of a point \(P \notin (\U \cup \ell_{\infty})\) and a line
\(\ell\) such that \(|\ell \cap \tau_{P}(\U)| = 3\), and demonstrate that our bound is sharp.
\begin{lemma}\label{lem:oval-parameterisation}
	Let \(y_{1} = 0\), then the points of the oval given by
	equation~\eqref{eq:oval} are
	\begin{equation}
		\left\{P_{u} = \left(\frac{z_{2}^{1 -
					\sigma/2}u^{\sigma}}{1 + u + u^{\sigma}}, \frac{z_{2}^{\sigma/2}(1 + u^{\sigma})}{1 + u + u^{\sigma}}\right)\,\middle|\,u \in \mathbb{F}_q \right\} \cup \left\{\left(z_{2}^{1 - \sigma/2},
		z_{2}^{\sigma/2}\right)\right\}.
	\end{equation}
\end{lemma}
\begin{proof}
	If \(y_{1} = 0\), then equation~\eqref{eq:oval} reduces to
	\begin{equation}
		\label{eq:oval-y1-zero}
		s^{\sigma + 2} + t^{\sigma} + st + z_{2} = 0
	\end{equation}
	Using the properties of \(\sigma\) described in
	Section~\ref{sec:preliminaries}, one can show the point
	\((z_2^{1-\sigma/2}, z_{2}^{\sigma/2})\) satisfies equation~\eqref{eq:oval-y1-zero}. Furthermore, the
	points
	\(\overline{P_u} = (z_{2}^{1 - \sigma/2}u^{\sigma}, z_{2}^{\sigma/2}(1 + u^{\sigma}), 1 + u + u^{\sigma})\),
	where \(u \in \mathbb{F}_q\), are projective points satisfying the following
	homogeneous equation
	\begin{equation}
		x^{\sigma + 2} + y^{\sigma}z^{2} + xyz^{\sigma} +
		z_{2}z^{\sigma + 2} = 0.
	\end{equation}
	Because \(\trace(u + u^{\sigma}) = 0\), and \(\trace(1) = 1\) when
	\(q = 2^{2e + 1}\),  we have \(u^{\sigma} + u+ 1 \neq 0\) for all
	\(u \in \mathbb{F}_q\). Thus, normalising so \(z = 1\), the points
	\(\overline{P_u}\) have the form \((s, t, 1)\) where \(s\) and \(t\) satisfy
	equation~\eqref{eq:oval-y1-zero}.
\end{proof}
\begin{corollary}\label{cor:oval-polynomial}
	Let \(y_{1} = 0\) and consider the points \(P_u\) as described in
	Lemma~\ref{lem:oval-parameterisation}. A point \(P_{u}\) lies on the conic given by
	equation~\eqref{eq:conic}, if and only if \(u\) is a root of the following polynomial
	\begin{equation}\label{final}
		a_{1}^{\sigma/2}u^{\sigma} + (z_{2}^{\sigma - 1} +
		\delta^{\sigma/2}z_{2} + z_{2}^{\sigma/2} +
		a_{1}^{\sigma/2})u^{2} + z_{2}^{\sigma/2}u +
		\delta^{\sigma/2} z_{2} + a_{1}^{\sigma/2}
	\end{equation}
\end{corollary}
\begin{proof}
	By directly substituting \(P_u\) into equation~\eqref{eq:conic} we have
	\begin{equation}
		\label{eq:oval-polynomial}
		(z_{2}^{2 - \sigma} + \delta z_{2}^{\sigma} + z_{2} +
		a_{1})u^{2\sigma} + z_{2}u^{\sigma} + a_{1}u^{2} +
		(\delta z_{2}^{\sigma} + a_{1}) = 0
	\end{equation}
	Raising both sides of equation~\eqref{eq:oval-polynomial} to the power of \(\sigma/2\)
	yields our result.
\end{proof}

\begin{theorem}\label{thm:bt-unital-line-meet-feet}
	Let \(U\) be a Buekenhout-Tits unital in \(\PG(2, q^{2})\). The feet of a
	point \(P \notin (\ell_{\infty} \cup U)\) meet a line \(\ell\) in at most four
	points. Moreover, there exists a line \(\ell\) and point \(P\) such that
	\(|\ell \cap \tau_{P}(U)| = k\) for each \(k \in \{0, 1, 2, 3,4\}\).
\end{theorem}
\begin{proof}
  By Theorem~\ref{cor:proj-equivalence} we may assume that \(U = \U\).
  The first part of the proof comes from Corollary~\ref{cor:rough-bound}. Let
  \(P = (1, y_1, z_2\epsilon)\). All lines through \(P\) meet \(\tau_P(U)\) in at
  most one point by definition, so it is clear that there exists lines \(\ell\) such
  that \(|\ell \cap \tau_P(U)|\) is zero or one. Because the points of \(\tau_P(U)\)
  are not collinear, there exists a pair of points \(Q, R \in \tau_P(U)\) such
  that the line \(QR\) does not contain \((0, 1, y_1)\). Hence, the line \(QR\)
  meets in precisely two points by Lemma~\ref{lem:easy-case}. 
  
  Now consider a line
  \(\ell\) with equation \((\delta + \epsilon)x + z = 0\) and let $P$ be the point $(1,0,\epsilon)$
(that is,  \(a_1=\delta, a_2=1, b_1=b_2=y_{1} = 0, z_{2} = 1\)). The number of points of \(\ell \cap \tau_P(U)\) is the
  same as the number of solutions to equations~\eqref{eq:sys-1-simple} and
  \eqref{eq:sys-2-simple}. By Lemma~\ref{lem:oval-parameterisation} the points
  \(P_u\) satisfying equation~\eqref{eq:sys-2-simple} lie on the
  conic~\eqref{eq:sys-1-simple} when
  \begin{equation}
	\delta^{\sigma/2}u^{\sigma} + u = u(\delta^{\sigma/2}u^{\sigma - 1} + 1) = 0,
  \end{equation}
  which has two roots as \(\sigma - 1\) is a permutation of \(\mathbb{F}_q\).
  It can also be shown that \((z_2^{1-\sigma/2}, z_2^{\sigma/2}) = (1, 1)\)
  satisfies both equations. Hence, the intersection of the feet of the point
  \((1, 0, \epsilon)\) and \(\ell\) has exactly three points.
  
  Finally, consider the point $P(1,0,\frac{1}{\delta^\sigma}\epsilon)$ and the line $\ell$ with dual coordinates $[\frac{1}{\delta}+\frac{1}{\delta^2}\epsilon,0,1]$. By Corollary \ref{cor:oval-polynomial}, the number of feet of $P$ on the line $\ell$ is the number of solutions to the equation \eqref{final}, where $a_1=\frac{1}{\delta}$ and $z_2=\frac{1}{\delta^\sigma}$ which is  \begin{equation}\label{final2}\frac{1}{\delta^{\sigma/2}}u^{\sigma} + (\frac{1}{\delta^{2-\sigma}}+\frac{1}{\delta})u^2+\frac{1}{\delta}u=0.
  \end{equation}
  Since equation~\eqref{final2} is a $\F_2$-linearised polynomial, and there are at most $4$ roots, we have that equation~\eqref{final} has $1,2,$ or $4$ roots. We will show that, under the condition $\trace(\delta)=1$, it has four roots.
  Multiplying equation~\eqref{final2} by $\delta$ yields $\delta^{1-\sigma/2}u^\sigma+(\delta^{\sigma-1}+1)u^2+u=0$ and now substituting $a=\delta^{\sigma-1}+1$ gives
  \begin{equation}\label{h1}(a^{\sigma/2}+1)u^\sigma+au^2+u=0.\end{equation}
We find that \(u = 0\) and \(u = \frac{1}{a^{1 + \sigma/2}}\) are solutions to equation~\eqref{h1}. Now consider
   \begin{equation}
     \label{eq:h2}
     u^{\sigma} + au^2 + 1 = 0.
   \end{equation}
   Any solution to equation~\eqref{eq:h2} also satisfies $(u^{\sigma} + au^2 + 1)^{\sigma / 2} + u^{\sigma} + a u^2 + 1 = 0$ which is precisely equation~\eqref{h1}.
 Multiply equation~\eqref{eq:h2} with $a^{\sigma+1}$, then we find $(a^{\sigma/2+1}u)^\sigma+(a^{\sigma/2+1}u)^2+a^{\sigma+1}=0$, and letting \(z = (a^{\sigma/2 + 1}u)^2\),
\begin{equation}\label{eq:transform-eq} z^{\sigma/2}+z+a^{\sigma+1}=0,\end{equation}
which is known (see \cite{menichetti}) to have solutions if and only if $\trace(a^{\sigma+1})=0$. As \(z = 0\) and \(z = 1\) are not solutions of equation~\eqref{eq:transform-eq}, no solutions of equation~\eqref{eq:transform-eq} correspond to the solutions \(u = 0\) or \(u = \frac{1}{a^{1 + \sigma/2}}\) of equation~\eqref{final2}. Furthermore, recall that equation~\eqref{final2} has $1,2$ or $4$ solutions and that we have assumed that $\trace(\delta)=1$. Since $\delta^{\sigma-1}=a+1$, it follows that $\delta=(a+1)^{\sigma+1}$ and $\trace(\delta)=\trace(a^{\sigma+1}+a^\sigma+a+1)=\trace(a^{\sigma+1})+\trace(1)=\trace(a^{\sigma+1})+1$. Hence, the conditions $\trace(\delta)=1$ and $\trace(a^{\sigma+1})=0$ are equivalent,  and we find exactly four roots to equation~\eqref{final2}.
\end{proof}

\printbibliography%

\noindent
{\bf Address of the authors:}\\
Jake Faulkner \texttt{jake.faulkner@pg.canterbury.ac.nz}\\
\noindent
Geertrui Van de Voorde \texttt{geertrui.vandevoorde@canterbury.ac.nz}\\

\noindent
School of Mathematics and Statistics\\
University of Canterbury\\
Private bag 4800\\
8140 Christchurch\\
New Zealand
\end{document}